\documentclass[11pt,letterpaper,review]{amsart}

\usepackage[utf8]{inputenc}
\usepackage[T1]{fontenc}
\usepackage{fullpage}
\usepackage{amsthm,amsmath,amssymb}
\usepackage{amsfonts}
\usepackage{graphicx}
\usepackage{floatrow}
\usepackage{hyperref}
\usepackage{amsrefs}   
\usepackage[left, mathlines]{lineno}
\usepackage[inline]{enumitem} 

\theoremstyle{plain}
\newtheorem{theorem}{Theorem}

\newtheorem{proposition}[theorem]{Proposition}

\newtheorem{lem}[theorem]{Lemma}
\newtheorem{corollary}[theorem]{Corollary}
\newtheorem{observation}[theorem]{Observation}

\newtheorem*{observation-sn}{Observation}

\theoremstyle{plain} 
\newcommand{\thistheoremname}{}
\newtheorem{genericthm}[theorem]{\thistheoremname}

\theoremstyle{remark}

\theoremstyle{plain}\newtheorem{definition}{Definition}

\newcommand{\cyl}[1]{\operatorname{cr}_{\ensuremath{\circledcirc}}(#1)}

\newcommand{\cicr}[2]{\operatorname{cr}_{\ensuremath{#1\circ}}(#2)}

\newcommand{\bk}[2]{\operatorname{bkcr}_{#1}(#2)}

\newcommand{\dd}{\ensuremath{\mathcal{D}}}
\newcommand{\ee}{\ensuremath{\mathcal{E}}}

\def\ignore#1{{}}

\title{\bf The complexity of computing the cylindrical \\and the
  \(t\)-circle crossing number of a graph}

\author[F. Duque]{Frank Duque}
\address{Instituto de Matem{\'a}ticas\\ Universidad de
  Antioquia. Medell{\'i}n, Colombia 050010}
\email{rodrigo.duque@udea.edu.co}

\author[H. Gonz{\'a}lez-Aguilar]{Hern{\'a}n Gonz{\'a}lez-Aguilar}
\address{Facultad de Ciencias\\ Universidad Aut{\'o}noma de San Luis
  Potos{\'i}. San Luis Potos{\'i}, M{\'e}xico 78290}
\email{hernan@fc.uaslp.mx}

\author[C. Hern{\'a}ndez-V{\'e}lez]{C{\'e}sar Hern{\'a}ndez-V{\'e}lez$^{\ast}$}
\address{Facultad de Ciencias\\ Universidad Aut{\'o}noma de San Luis
  Potos{\'i}. San Luis Potos{\'i}, M{\'e}xico 78290}
\email{cesar.velez@uaslp.mx}
\thanks{$^{\ast}$Partially supported by FAPESP (Proc. 2013/03447-6) and CNPq (Proc. 456792/2014-7)}

\author[J. Lea{\~n}os]{Jes{\'u}s Lea{\~n}os$^{\dag}$}
\address{Unidad Acad{\'e}mica de Matem{\'a}ticas\\ Universidad Aut{\'o}noma de
  Zacatecas. Zacatecas, M{\'e}xico 9800}
\email{jleanos@matematicas.reduaz.mx}
\thanks{$^{\dag}$Partially supported by CONACyT Grant 179867}

\author[C. Medina]{Carolina~Medina$^{\S}$}
\address{Instituto de F{\'i}sica\\ Universidad Aut{\'o}noma de San Luis
  Potos{\'i}. San Luis Potos{\'i},  M{\'e}xico 78290}
\email{cmedina@ifisica.uaslp.mx}
\thanks{$^{\S}$Partially supported by CONACyT Grant 222667}

\begin{document}


\maketitle

\begin{abstract}

A plane drawing of a graph is {\em cylindrical} if there exist two
concentric circles that contain all the vertices of the graph, and no
edge intersects (other than at its endpoints) any of these circles. The
{\em cylindrical crossing number} of a graph \(G\) is the minimum number
of crossings in a cylindrical drawing of \(G\). In his influential
survey on the variants of the definition of the crossing number of a graph,
Schaefer lists the complexity of computing the cylindrical crossing
number of a graph as an open question. In this paper we settle this by
showing that this problem is NP-complete. Moreover, we show an analogous
result for the natural generalization of the cylindrical crossing
number, namely the \(t\)-{\em circle crossing number}.\\

\end{abstract}


\section{Introduction}\label{sec:introduction}

This work is motivated by a question posed by Marcus Schaefer in his survey on
the variants of the definition of the crossing number of a
graph. In~\cite{schaefer2014}, Schaefer listed as open the problem of the
complexity of computing the cylindrical crossing number of a graph. We recall
that a {\em cylindrical drawing} of a graph \(G\) is a plane drawing where all
the vertices are in two concentric cycles, and no circle is intersected by the
interior of an edge. The {\em cylindrical crossing number} \(\cyl{G}\) of a
graph \(G\) is the minimum number of crossings in a cylindrical drawing of
\(G\).

The concept of a cylindrical drawing is motivated by a family of graph
drawings of the complete graph \(K_n\), originally conceived by the
British artist Anthony Hill. As narrated in the lively account given
in~\cite{earlyhistory}, Hill's construction produces drawings of \(K_n\)
that are cylindrical, according to the definition above, and have
exactly
\(Z(n):= \frac{1}{4} \big\lfloor \frac{n}{2} \big\rfloor \big\lfloor
\frac{n-1}{2} \big\rfloor \big\lfloor \frac{n-2}{2} \big\rfloor
\big\lfloor \frac{n-3}{2} \big\rfloor\) crossings. It is a long-standing
conjecture that the crossing number of \(K_n\) is \(Z(n)\), for every
\(n\ge 3\)~\cite{hh}. In~\cite{abrego2014}, \'Abrego et al.~proved that
\(\cyl{K_n}=Z(n)\), for every \(n\ge 3\).

Let \(\dd\) be a plane drawing of a graph \(G\). We say that a Jordan
curve \(\rho\) (that is, a simple closed curve) is {\em clean} (with
respect to \(\dd\)) if the interior of no edge intersects \(\rho\). Now
suppose that there are two clean disjoint circles  with respect
to \(\dd\), say \(\rho_1\) and \(\rho_2\),
such that every vertex of \(G\) is in \(\rho_1\cup \rho_2\). Note that
not only concentricity is not assumed, but also it is not required that
the disk bounded by one of these circles contains the other circle. It
is a straightforward exercise in plane topology that there is a
cylindrical drawing \(\dd'\) with the same cellular structure as
\(\dd\); in particular, \(\dd'\) has the same number of crossings as
\(\dd\). Thus for crossing number purposes it is totally valid to adopt
the following definition of a cylindrical drawing.

\begin{definition}[Equivalent definition of cylindrical drawing]
  A plane drawing of a graph \(G\) is {\em cylindrical} if there exists two  
  disjoint clean circles \(\rho_1,\rho_2\) such that every vertex of
  \(G\) is in \(\rho_1\cup \rho_2\).
\end{definition}

The advantage of adopting this definition of a cylindrical drawing is
that it allows us to generalize this notion to an arbitrary number of
circles, as follows. We should mention that the term ``\(t\)-circle
drawing'' has been suggested by \'Eva Czabarka and Marcus Schaefer
(private communication).

\begin{definition}[\(t\)-circle drawing and \(t\)-circle crossing
  number]\label{def:tcircle}
  Let \(t\ge 1\) be an integer. A plane drawing of a graph \(G\) is a
  \(t\)-{\em circle drawing} if there exist \(t\) pairwise disjoint
  clean circles \(\rho_1,\ldots,\rho_t\) such that every vertex of \(G\)
  is in \(\rho_1\cup \cdots\cup \rho_t\). The \(t\)-{\em circle crossing
    number} \(\cicr{t}{G}\) of a graph \(G\) is the minimum number of
  crossings in a \(t\)-circle drawing of \(G\).
\end{definition}

Thus a cylindrical drawing is simply a \(2\)-{circle} drawing. Moreover,
for \(t=1\) there is an immediate connection with \(2\)-page
drawings. We recall that a \(2\)-{\em page drawing} of a graph is a
drawing in which the vertices lie on the \(x\)-axis, and each edge is
contained (except for its endpoints) either in the upper halfplane, or
in the lower halfplane. A straightforward argument shows that a
\(1\)-circle drawing can be transformed into a \(2\)-page drawing with the
same cellular structure.

Thus the \(1\)-circle crossing number of a graph coincides with its
\(2\)-page crossing number, and the \(2\)-circle crossing number of a
graph is the same as its cylindrical crossing number. The \(3\)-circle
crossing number is related to the {\em pair of pants crossing
  number}~\cite{schaefer2014}, but these are different notions, since in
the latter it is required that none of the disks bounded by the circles
contains another circle, and that no edge intersects the interior of any
of these disks.

For the arguments we will use in this paper, it will be useful to relax
the condition that the clean Jordan curves in
Definition~\ref{def:tcircle} need to be circles:

\begin{definition}[\(t\)-curve drawing and \(t\)-curve crossing
  number]\label{def:tcurve}
  Let \(t\ge 1\) be an integer. A plane drawing of a graph \(G\) is a
  \(t\)-{\em curve drawing} if there exist \(t\) pairwise disjoint clean
  Jordan curves \(\rho_1,\ldots,\rho_t\) such that every vertex of \(G\)
  is in \(\rho_1\cup \cdots\cup \rho_t\). The \(t\)-{\em curve crossing
    number} of a graph \(G\) is the minimum number of crossings in a
  \(t\)-curve drawing of \(G\).
\end{definition}

It follows from the Jordan-Sch\"onflies theorem that if \(\dd\) is a
\(t\)-curve drawing of a graph \(G\), then there is a self-homeomorphism
of the plane that takes \(\dd\) to a \(t\)-circle drawing. In
particular, for any graph \(G\), its \(t\)-circle crossing number and
its \(t\)-curve crossing number are the same. Thus the difference
between these notions is rather cosmetic. On the other hand, as we
hinted above, the advantage of dealing with \(t\)-curve drawings instead
of \(t\)-circle drawings is that being able to work with arbitrary
Jordan curves, instead of exclusively with circles, makes our arguments
simpler.

As we mentioned above, our motivation in this work is to settle the
complexity of computing the cylindrical crossing number of a graph,
that is, the complexity of the decision problem {\sc
  CylindricalCrossingNumber}: ``given a graph \(G\) and an integer
\(k\), is \(\cyl{G}\le k\)''?

As it happens, with very little additional effort we can settle the
complexity of the decision problem {\sc \(t\)-curveCrossingNumber}, that
considers a fixed integer \(t\) and asks ``given a graph \(G\) and an
integer \(k\), is the \(t\)-curve crossing number of \(G\) at most
\(k\)?''.

\begin{theorem}\label{thm:maintheorem}
  For each fixed integer \(t\ge 2\), \(t\)-{\sc curveCrossingNumber} is
  NP-complete.
\end{theorem}

As we mentioned above, the \(t\)-circle crossing number of a graph and
its \(t\)-curve crossing number are the same. Thus this settles the
complexity of the decision problem {\sc \(t\)-circleCrossingNumber},
that considers a fixed integer \(t\) and asks ``given a graph \(G\) and
an integer \(k\), is \(\cicr{t}{G}\le k\)?''. Since the \(2\)-circle
crossing number of a graph is its cylindrical crossing number, this
settles in particular the complexity of computing the cylindrical
crossing number. For completeness, we state these observations formally:

\begin{corollary}\label{cor:cylindrical}
  For each fixed integer \(t\ge 2\), \(t\)-{\sc circleCrossingNumber} is
  NP-complete. In particular, {\sc CylindricalCrossingNumber} is
  NP-complete.
\end{corollary}

Before proceeding to the proof of Theorem~\ref{thm:maintheorem}
(Section~\ref{sec:proofmain}), we establish in the next section a result
on plane triangulations that are minimal with respect to having a
$t$-curve embedding.


\section{Minimal \texorpdfstring{\(t\)}{t}-curve embeddings}\label{sec:aux}

An essential ingredient in the proof that \(t\)-{\sc
  curveCrossingNumber} is NP-hard is the existence of plane
triangulations that are minimal with respect to having a \(t\)-curve
embedding. Our aim in this section is to establish this result
(Lemma~\ref{lem:tri} below). We will need the following statement.

\begin{proposition}\label{pro:factor}
  Let \(G\) be a maximal planar graph, and let \(t\) be a positive
  integer. Suppose that \(G\) has a \(t\)-curve embedding. Then there is
  a collection \(\{H_1,\ldots,H_t\}\) of pairwise disjoint subgraphs
  of \(G\) with the following properties: (i) if \(H_i\) has at least
  \(3\) vertices for some \(i\in\{1,\ldots,t\}\), then \(H_i\) is a
  cycle; and (ii) \(\bigcup_{i=1}^t H_i\) contains all the vertices of
  \(G\).
\end{proposition}

\begin{proof}
  Let \(\ee\) be a \(t\)-curve embedding of \(G\), and let
  \(\rho_1,\ldots,\rho_t\) be the underlying $t$ clean Jordan curves of
  \(\ee\). Let \(i\in\{1,\ldots,t\}\). If \(\rho_i\) does not contain
  any vertex, then we let \(H_i\) be the null graph. If \(\rho_i\)
  contains at least one vertex, let \(v_1,\ldots,v_{m_i}\) be the
  vertices on \(\rho_i\), in the (cyclic) order in which
  they appear in \(\rho_i\). If \(m_i=1\), then we let \(H_i\) be the
  subgraph of \(G\) that consists only of the vertex \(v_1\). If
  \(m_i\ge 2\), we proceed as follows.

  For \(j=1,\ldots,m_i\), there is a subarc of \(\rho_i\) whose endpoints are
  \(v_j\) and \(v_{j+1}\) (indices are taken modulo \(m_i\)), and that is
  otherwise disjoint from \(G\). This implies that for \(j=1,\ldots,m_i\), there
  is a face incident with \(v_j\) and \(v_{j+1}\). Since \(G\) is maximal
  planar, \(\ee\) is a plane triangulation and it is the unique plane embedding
  of \(G\) (up to homeomorphism). Therefore the existence of a face incident
  with \(v_j\) and \(v_{j+1}\) implies that \(v_j\) and \(v_{j+1}\) are
  adjacent.

  If \(m_i=2\), then we let \(H_i\) be the subgraph of \(G\) that consists
  of the vertices \(v_1\) and \(v_2\), and the edge joining them. If
  \(m_i\ge 3\), then \(v_1v_2\ldots v_{m_i} v_1\) is a cycle \(C_i\) of \(G\),
  and we let \(H_i=C_i\).

  Since each vertex of \(G\) is contained in a curve in
  \(\{\rho_1,\ldots,\rho_t\}\), and these curves are pairwise
  disjoint, it follows that the collection \(\{H_1,\ldots,H_t\}\)
  satisfies the required conditions.
\end{proof}

\begin{lem}\label{lem:tri}
  For every \(t \ge 2\) there is a \(3\)-connected simple graph \(G_t\)
  such that (i) \(G_t\) triangulates the plane; (ii) \({G_t}\) has a
  \(t\)-curve embedding; and (iii) \(G_t\) has no \((t-1)\)-curve embedding.
\end{lem}

\begin{proof} 
  The heart of the proof is the existence of triangulations whose
  longest cycles are relatively small. Following Chen and
  Yu~\cite{chen2002}, let \(T_1,T_2,\ldots\) be the family of
  triangulations constructed as follows. First, \(T_1\) is the plane
  triangulation induced by  \(K_4\). Now, \(T_{i+1}\) is
  constructed from \(T_i\), for \(i=1,2,\ldots\), as follows: in each inner face of~\(T_i\),
  add one new vertex and join it to the vertices of~\(T_i\) incident
  with the face containing it. We refer the reader to
  Figure~\ref{fig:fig1}.

\begin{figure}[ht] 
\includegraphics[width=15.0cm]{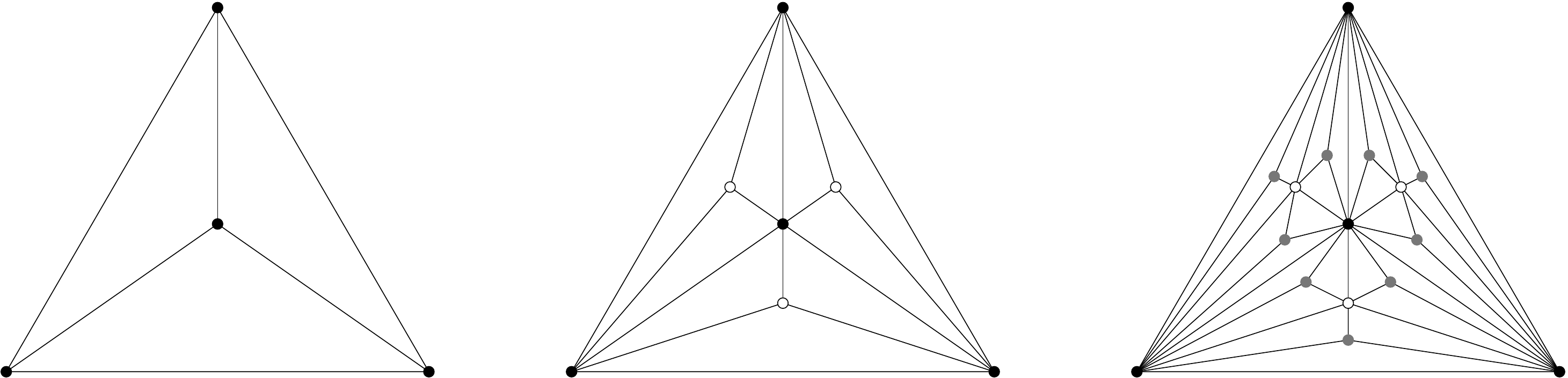}
\caption{On the left hand side we have the triangulation \(T_1\). For
  each inner face \(T\) of \(T_1\), we add a (white) vertex inside
  \(T\) and join it with edges to the three vertices incident to \(T\); the
  result is the middle triangulation \(T_2\). We obtain \(T_3\) (right
  hand side) similarly: for each inner face \(T\) of \(T_2\), we add
  a (grey) vertex inside \(T\), and join it with edges to the three
  vertices incident to \(T\).}
\label{fig:fig1}
\end{figure}

In~\cite{chen2002} it is proved that, for \(i \ge 1\), the length of the longest
cycle of \(T_i\) is less than \(\frac{7}{2}|V(T_i)|^{\log_3 2}\). Now let \(j\)
be an integer large enough such that
\(\frac{7}{2}|V(T_j)|^{\log_3 2}\cdot (t-1) < |V(T_j)|\). Toward a
contradiction, suppose that \(T_j\) has a \((t-1)\)-curve embedding, and let
\(\{H_1,\ldots,H_{t-1}\}\) be the subgraphs of \(T_j\) guaranteed by
Proposition~\ref{pro:factor}. Since \(V(T_j)=\bigcup_{i=1}^{t-1} V(H_i)\), it
follows that there is some \(H_i\) such that
\(|V(H_i)| > \frac{7}{2}|V(T_j)|^{\log_3 2}\). Since
\(\frac{7}{2}s^{\log_3 2} \ge 3\) for every \(s \ge 1\), it follows from
Proposition~\ref{pro:factor} that \(H_i\) must be a cycle, contradicting that
the length of the longest cycle of \(T_j\) is less than
\(\frac{7}{2}|V(T_j)|^{\log_3 2}\). Thus \(T_j\) has no \((t-1)\)-curve
embedding.

In the previous paragraph we have shown that the family of graphs which are no
\((t-1)\)-curve embeddable is not empty. Now we will choose the required graph
from such a family.  Let \(m\) be the least integer such that \(T_m\) has no
\((t-1)\)-curve embedding. Note that \(m\ge 3\), since \(T_2\) has a \(1\)-curve
embedding, and thus a \((t-1)\)-curve embedding for every \(t\ge 2\). By the
minimality of \(m\), \(T_{m-1}\) has a \((t-1)\)-curve embedding. Let
\(Q_1:=T_{m-1}, Q_2, \ldots, Q_k:=T_{m}\) be a sequence of triangulations
(subtriangulations of \(T_{m}\)) such that \(Q_{i+1}\) is obtained from \(Q_i\)
by adding a new vertex and its three incident edges, for
\(i \in \{1, \ldots, k-1\}\). Let \(\ell\) be the largest integer such that
\(Q_\ell\) is a \((t-1)\)-curve embedding. Let \(v\) be the vertex that gets
added (together with its three incident edges) to \(Q_\ell\), in order to get
\(Q_{\ell+1}\).

The maximality of \(\ell\) implies that \(Q_{\ell+1}\) is not a
\((t-1)\)-curve embedding, and we claim that \(Q_{\ell+1}\) is a
\(t\)-curve embedding. To see this, let \(x,y,z\) be the three vertices
adjacent to \(v\) in \(Q_{\ell+1}\). Thus \(x,y,z\) form a \(3\)-cycle,
which bounds the face \(f\) in \(Q_\ell\) in which \(v\) is placed. Let
\(\rho_1,\ldots,\rho_{t-1}\) be clean Jordan curves that witness the
\((t-1)\)-curve embeddability of \(Q_\ell\). It is easy to see that if
one of these Jordan curves intersects \(f\) then we can slightly perturb
it so that it also intersects \(v\). But this is impossible, since then
\(Q_{\ell+1}\) would be a \((t-1)\)-curve embedding. Thus none of
\(\rho_1,\ldots,\rho_{t-1}\) intersects \(v\) or its incident edges, and
so they are also clean Jordan curves in \(Q_{\ell+1}\). We now draw in a
small neighborhood of \(v\) a clean Jordan curve \(\rho_t\) that only
contains \(v\), so that \(\rho_1,\ldots,\rho_t\) is a
collection of pairwise disjoint clean Jordan curves that contain all the
vertices of \(Q_{\ell+1}\). Therefore \(Q_{\ell+1}\) is a \(t\)-curve
embedding, as claimed.

Let \(G_t\) be the underlying graph of the triangulation
\(Q_{\ell+1}\). It is readily checked that \(G_t\) is \(3\)-connected
and simple, and \(Q_{\ell+1}\) witnesses that \(G_t\) triangulates the
plane, and that \(G_t\) has a \(t\)-curve embedding. Since \(G_t\) is
\(3\)-connected it follows that \(Q_{\ell+1}\) is 
its unique embedding (up to isomorphism) in the plane. Since \(Q_{\ell+1}\) is not a
\((t-1)\)-curve embedding, it follows that \(G_t\) does not have a
\((t-1)\)-curve embedding.
\end{proof}


\section{Proof of Theorem~\ref{thm:maintheorem}}\label{sec:proofmain}

First we prove membership in NP, and then we prove NP-hardness.

\vglue 0.3 cm
\noindent{\sc (A)} {\sc \(t\)-curveCrossingNumber} is in NP.
\vglue 0.3 cm

\begin{proof}
  A drawing \(\dd\) of a graph \(G\) with at most \(k\) crossings may be
  described combinatorially, with an amount of information bounded by a
  polynomial function of \(|V(G)|+|E(G)|+k\), by giving the cellular
  structure of \(\dd\). Now a collection \(R\) of \(t\) clean Jordan
  curves with respect to \(\dd\) can be also described combinatorially, simply by
  regarding each of these curves as the edge set of a cycle that gets
  added to \(\dd\). (We remark that for this purpose we regard a graph
  that consists of a pair of vertices joined by two parallel edges, or
  of a vertex with a loop-edge, as a cycle.) We let \(\dd'\) denote the
  drawing that is obtained from \(\dd\) by adding (the edges of) these
  \(t\) cycles, which we colour blue to help comprehension.

  Thus the fact that \(G\) has a \(t\)-curve drawing with at most \(k\)
  crossings can be attested by the existence of such a drawing \(\dd'\),
  with the properties that the \(t\) blue cycles are pairwise disjoint,
  and each vertex of \(G\) is contained in a blue cycle. We finally note
  that such a drawing \(\dd'\) can be described combinatorially with an
  amount of information bounded by a polynomial function of
  \(|V(G)|+|E(G)|+k+t\), and that it can be verified in polynomial time
  that \(\dd'\) satisfies the required properties.
\end{proof}

\vglue 0.3 cm
\noindent{\sc (B)} {\sc \(t\)-curveCrossingNumber} is NP-hard.
\vglue 0.3 cm

\begin{proof}
  Let \(t \ge 2\) be fixed. Let \(G\) be a graph and let \(G'\) be the
  disjoint union of \(G\), a graph \(G_t\) that satisfies the conditions
  in Lemma~\ref{lem:tri}, and \(k\) disjoint copies of \(K_{3,3}\). It
  was proved in~\cite{chung1987} that testing if a graph has a
  \(2\)-page embedding is NP-complete. Since the size of \(G'\) is
  bounded by a polynomial function of \(|V(G)|+|E(G)|+k\) (the size of
  \(G_t\) is a constant, for each fixed \(t\)), it follows that to prove
  {\sc (B)} it suffices to show that \(G\) has a \(2\)-page embedding if
  and only if \(G'\) has a \(t\)-curve drawing with at most \(k\)
  crossings.

  Suppose that \(G\) has pagenumber \(2\). Let \(\ee\) be a \(t\)-curve
  embedding of \(G_t\). Let \(\rho\) be one of the \(t\) clean Jordan
  curves that witness that \(\ee\) is a \(t\)-curve embedding, and let
  \(p\) be a point on \(\rho\) that is not a vertex of $G$. Let \(\delta\) be a
  disk with center \(p\), small enough so that \(\delta\) does not
  intersect any vertex or edge of \(G_t\). Then we can embed \(G\) in
  the interior of \(\delta\), with the vertices lying on
  \(\rho\cap \delta\). Moreover, since the \(2\)-page crossing number of
  \(K_{3,3}\) is \(1\), it follows that we can also draw the
  \(k\) copies of \(K_{3,3}\) in the interior of $\delta$, with one crossing per copy, so that all
  the vertices lie on \(\rho\cap\delta\). This yields a \(t\)-curve
  drawing of~\(G'\) with at most (actually, exactly) \(k\) crossings.
 
  For the other direction, suppose that \(\dd'\) is a \(t\)-curve
  drawing of \(G'\) with at most \(k\) crossings. We note that each copy
  of \(K_{3,3}\) must contribute with exactly \(1\) crossing. Thus it
  follows that if we let \(G''\) denote the disjoint union \(G\) and
  \(G_t\), then the restriction \(\ee''\) of \(\dd'\) to \(G''\) is a
  \(t\)-curve embedding.

  Let \(R:=\{\rho_1,\ldots,\rho_t\}\) be a set of clean Jordan
  curves that witness that \(\ee''\) is a \(t\)-curve embedding. We let
  $\ee_t$ denote the restriction of $\ee''$ to $G_t$. Then obviously the
  collection $R$ witnesses that $\ee_t$ is a $t$-curve embedding.

\vglue 0.3 cm
\noindent{\sc Claim.} {\sl Let $f$ be any face of $\ee_t$. Then there is
  at most one curve in $R$ that intersects $f$.}
\vglue 0.3 cm

\begin{proof}
  Let $f$ be any face of $\ee_t$. By Lemma~\ref{lem:tri}, every face in
  an embedding of \(G_t\) is a triangle, and so \(f\) is bounded by a
  \(3\)-cycle \(C\). Let \(u,v,w\) be the vertices of \(C\).

  To prove the claim, first note that at most three curves in \(R\) can
  intersect \(f\); this follows simply because \(C\) has exactly three
  vertices, and the curves in \(R\) are pairwise disjoint and clean with
  respect to \(\ee_t\). Suppose that exactly two curves
  \(\rho_i,\rho_j\) in \(R\) intersect \(f\). Since the curves in \(R\)
  are pairwise disjoint, then it is not possible that each of \(\rho_i\)
  and \(\rho_j\) intersects two vertices of \(C\). Thus at least one of
  these curves, say \(\rho_i\), must be a loop based on a vertex of
  \(C\), say \(u\). The other curve \(\rho_j\) either contains both
  \(v\) and \(w\), or exactly one of them. Suppose first that \(\rho_j\)
  contains both \(v\) and \(w\). Thus the scenario is as depicted on the
  left side of Figure~\ref{fig:fig2}.

\begin{figure}[ht!]
\centering
\scalebox{1.0}{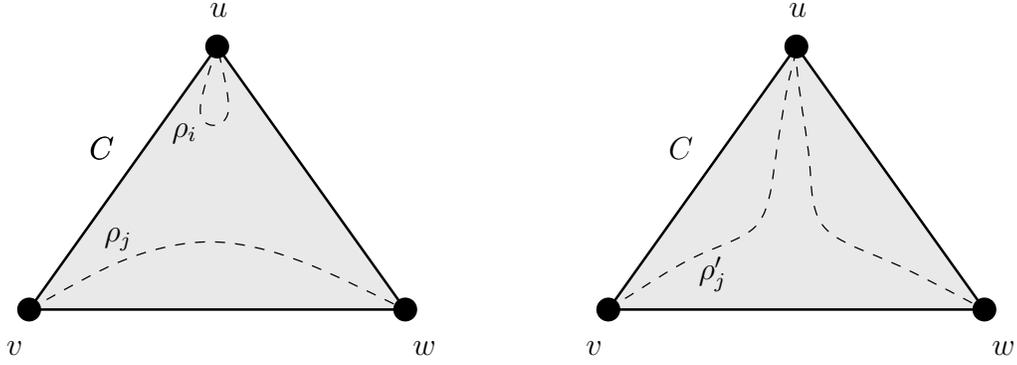}
\caption{The curve \(\rho_j\) intersects the (shaded) face $f$, and
  contains $v$ and $w$. Since the clean curve $\rho_i$ contains $u$,
  then $\rho_i\setminus\{u\}$ must be contained in $f$. In this case we
  can replace these two curves by a single curve \(\rho_j'\) that
  contains \(u,v\), and \(w\), as shown on the right-hand side.}
\label{fig:fig2}
\end{figure}

\begin{figure}[ht!]
\centering
\scalebox{1.0}{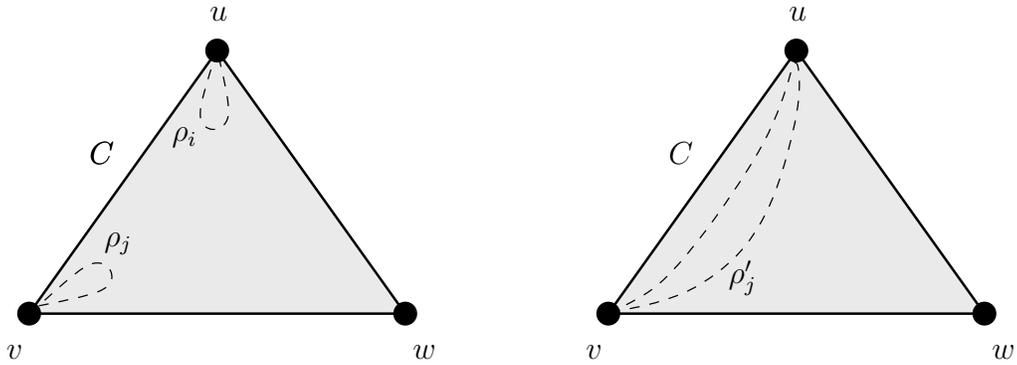}
\caption{The curve $\rho_i$ contains $u$, and is otherwise contained in
  $f$. The curve $\rho_j$ contains $v$, and is otherwise contained in
  $f$. In this case \(\rho_j\) can be re-routed inside \(f\), as
  illustrated on the right-hand side, so that the result is a clean
  Jordan curve \(\rho_j'\) that contains both \(u\) and $v$.}
\label{fig:fig3}
\end{figure}

We can then remove \(\rho_i\), and reroute the part of \(\rho_j\) inside
\(f\), so that the resulting curve \(\rho_j'\) contains \(v,u\), and
\(w\), as illustrated on the right side of
Figure~\ref{fig:fig2}. Hence
\((R\setminus\{\rho_i,\rho_j\})\cup \rho_j'\) is a set of \(t-1\)
pairwise disjoint clean Jordan curves whose union contains all the
vertices of \(G_t\). Therefore \(G_t\) has a \((t-1)\)-curve embedding,
contradicting (iii) in Lemma~\ref{lem:tri}. Now, if \(\rho_j\) contains
exactly one of \(v\) and \(w\) (say \(v\), without loss of generality),
then the scenario is as shown on the left side of
Figure~\ref{fig:fig3}. In this case we can replace \(\rho_i\) and
\(\rho_j\) by a curve \(\rho_j'\) that contains both \(u\) and \(v\) (as
in the right side of Figure~\ref{fig:fig3}). Thus
\((R\setminus\{\rho_i,\rho_j\})\cup \rho_j'\) is a set of \(t-1\)
pairwise disjoint clean Jordan curves whose union contains all the
vertices of \(G_t\), again contradicting (iii) in Lemma~\ref{lem:tri}.

\begin{figure}[ht!]
\centering
\scalebox{1.0}{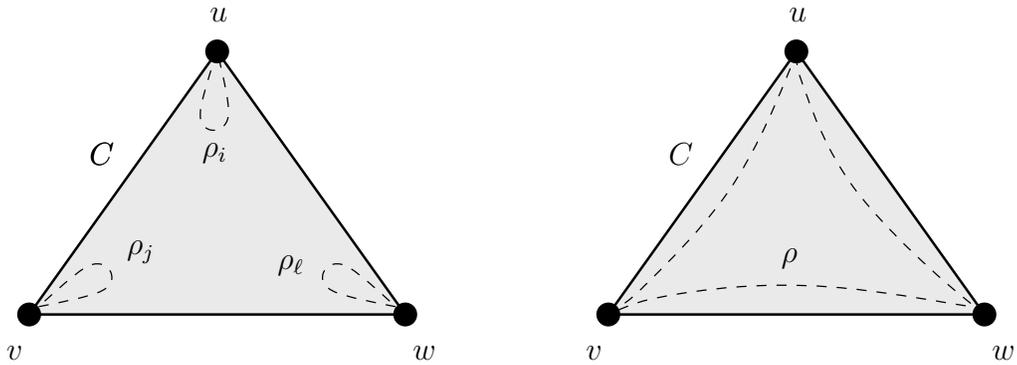}
\caption{If the clean Jordan curves $\rho_i, \rho_j, \rho_\ell$ contain
  $u,v$, and $w$, respectively, and each of these curves intersects $f$,
  then $\rho_i\setminus\{u\}, \rho_j\setminus\{v\}$, and
  $\rho_\ell\setminus\{w\}$ are contained in $f$, as shown in the
  left-hand side figure. These three curves can then be replaced by a
  single curve \(\rho\) that contains $u,v$, and $w$.}
\label{fig:fig4}
\end{figure}

In the remaining case, exactly three curves \(\rho_i,\rho_j,\rho_\ell\)
intersect \(f\). In this case each of these curves must contain exactly
one of \(u,v\), and \(w\), as illustrated on the left side of
Figure~\ref{fig:fig4}. We can then replace these three curves by a curve
\(\rho\) contained in \(f\), as shown on the right side of
Figure~\ref{fig:fig4}. Thus
\((R\setminus\{\rho_i,\rho_j,\rho_\ell\})\cup \rho\) is a set of \(t-2\)
pairwise disjoint clean Jordan curves whose union contains all the
vertices of \(G_t\). Hence \(G_t\) has a \((t-2)\)-curve embedding, (and
therefore, a \((t-1)\)-curve embedding), 
contradicting (iii) in Lemma~\ref{lem:tri}.
\end{proof}

Since $G$ and $G_t$ are disjoint, it follows that there is a face $f$ of
$\ee_t$ such that, in $\ee''$, $G$ is drawn inside $f$. Thus it follows
that some curve in \(R\) must intersect \(f\).

From the Claim, there is exactly one curve $\rho_m$ in \(R\) that
intersects \(f\). Since \(G\) is contained in \(f\), it follows that all
the vertices of \(G\) are contained in \(\rho_m\). Since \(\rho_m\) is
clean in \(\ee''\), it follows that \(\rho_m\) does not intersect any
edge of \(G\). Thus \(\rho_m\) witnesses that the restriction of $\ee''$
to \(G\) is a \(1\)-curve embedding. Hence we are done, since \(G\) has
a \(1\)-curve embedding if and only if it has a \(2\)-page embedding.
\end{proof}


\section{Concluding remarks}\label{sec:concludingremarks}

It follows from the proof of Theorem~\ref{thm:maintheorem} that, for
each fixed \(t\ge 2\), even the problem of deciding whether a given
graph admits a \(t\)-curve embedding, is already NP-complete. As we have
observed, this is also true for \(t=1\), as testing if a graph has
pagenumber \(2\) (which is equivalent to testing if it has a \(1\)-curve
embedding) is NP-complete.

We recall that a {\em \(p\)-page book} consists of \(p\) halfplanes (the
{\em pages}) whose boundaries lie on a common line ({\em the spine}). In
a {\em \(p\)-page drawing}, all the vertices lie on the spine, and each
edge (except for its endpoints) lies on a single
page~\cite{bernhart1979}. The \(p\)-{\em page crossing number}
\(\bk{p}{G}\) of a graph \(G\) is the minimum number of crossings in a
\(p\)-page drawing of \(G\)~\cite{sha}.

In the Book Crossing Number entry in~\cite{schaefer2014}, Schaefer
mentions that testing if a graph \(G\) satisfies \(\bk{p}{G}=0\) is
NP-complete, for every integer \(p\ge 2,p \neq 3\) (the case \(p=3\)
remains open). We note that analogous arguments to those we used in part
(B) of the proof of Theorem~\ref{thm:maintheorem} can be used to prove
the following.

\begin{observation}\label{obs:book}
  Let \(p\) be a fixed integer such that \(p\ge 2\) and \(p\neq
  3\). Then the decision problem ``given a graph \(G\) and an integer
  \(k\), is \(\bk{p}{G}\le k\)?'' is NP-complete.
\end{observation}

It is reasonable to argue that, alternatively to the definition of a
\(t\)-circle drawing, we could obtain a generalization of the definition
of a cylindrical drawing by asking that the vertices are contained in
\(t>2\) clean concentric circles. To illustrate an issue with such a
definition, let us consider drawings of the complete graph in which the
vertices are placed on three clean concentric circles. Then there cannot
be a vertex in the inner circle and a vertex in the outer circle, as
then an edge joining these two vertices would necessarily cross the
middle circle. Thus either all the vertices must lie in the union of the
middle circle and the outer circle, or in the union of the middle circle
and the inner circle. That is, any such drawing of the complete graph is
necessarily cylindrical. Thus, for general graphs, such an alternative
definition of a \(t\)-circle drawing is not really more general than the
definition of a cylindrical embedding. On the other hand, if one
slightly relaxes the condition that the interior of no edge intersects a
circle, then we arrive to the radial crossing
number~\cites{radial1,radial2,schaefer2014} (see also the related notion
of the cyclic level crossing number~\cite{cyc}).



\begin{bibdiv}
\begin{biblist}

\bib{abrego2014}{article}{ 
  author={{\'A}brego, Bernardo M.}, 
  author={Aichholzer, Oswin}, 
  author={Fern{\'a}ndez-Merchant, Silvia}, 
  author={Ramos, Pedro},
  author={Salazar, Gelasio}, 
  title={Shellable drawings and the cylindrical crossing number of \(K_n\)}, 
  journal={Discrete Comput. Geom.}, 
  volume={52},
  date={2014}, 
  number={4}, 
  pages={743--753}, 
}

\bib{radial1}{article}{
   author={Bachmaier, Christian},
   title={A radial adaptation of the Sugiyama framework for visualizing
hierarchical information},
   journal={IEEE Trans. Vis. Comput. Graph},
   volume={13},
   date={2007},
   number={3},
   pages={583--594},
}

\bib{cyc}{article}{
   author={Bachmaier, Christian},
   author={Brandenburg, Franz J.},
   author={Brunner, Wolfgang},
   author={H\"ubner, Ferdinand},
   title={Global \(k\)-level crossing reduction},
   journal={J. Graph Algorithms Appl.},
   volume={15},
   date={2011},
   number={5},
   pages={631--659},
}


\bib{earlyhistory}{article}{
   author={Beineke, Lowell},
   author={Wilson, Robin},
   title={The early history of the brick factory problem},
   journal={Math. Intelligencer},
   volume={32},
   date={2010},
   number={2},
   pages={41--48},
}

\bib{bernhart1979}{article}{ 
  author={Bernhart, Frank}, 
  author={Kainen, Paul C.},
  title={The book thickness of a graph}, 
  journal={J. Combin. Theory Ser. B},
  volume={27}, 
  date={1979}, 
  number={3}, 
  pages={320--331}, 
}



\bib{chen2002}{article}{ 
  author={Chen, Guantao}, 
  author={Yu, Xingxing},
  title={Long cycles in 3-connected graphs}, 
  journal={J. Combin. Theory Ser. B},
  volume={86}, 
  date={2002}, 
  number={1}, 
  pages={80--99}, 
}


\bib{chung1987}{article}{ 
  author={Chung, Fan R. K.}, 
  author={Leighton, Frank Thomson}, 
  author={Rosenberg, Arnold L.}, 
  title={Embedding graphs in books: a layout problem with applications to VLSI design},
  journal={SIAM J. Algebraic Discrete Methods}, 
  volume={8}, 
  date={1987}, 
  number={1}, 
  pages={33--58},
}

\bib{hh}{article}{
   author={Harary, Frank},
   author={Hill, Anthony},
   title={On the number of crossings in a complete graph},
   journal={Proc. Edinburgh Math. Soc. (2)},
   volume={13},
   date={1962/1963},
   pages={333--338},
}

\bib{radial2}{article}{
	author={Northway, Mary L.},
	title={A method for depicting social relationships obtained by sociometric testing},
	journal={Sociometry},
	volume={3},
	date={1940},
	number={2},
	pages={144--150},
}

\bib{schaefer2014}{article}{ 
  author={Schaefer, Marcus}, 
  title={The graph crossing number and its variants: A survey}, 
  journal={Electron. J. Combin.}, 
  date={2014}, 
  pages={Dynamic Survey 21, 100}, 
}

\bib{sha}{article}{
   author={Shahrokhi, Farhad},
   author={Sz\'ekely, L\'aszl\'o A.},
   author={S\'ykora, Ondrej},
   author={Vr\v to, Imrich},
   title={The book crossing number of a graph},
   journal={J. Graph Theory},
   volume={21},
   date={1996},
   number={4},
   pages={413--424},
}


\end{biblist}
\end{bibdiv}
\end{document}